\documentclass[a4paper,10pt]{article}
\usepackage[english]{babel}
\usepackage[latin1]{inputenc}
\usepackage[T1]{fontenc}
\usepackage{textcomp} 
\usepackage{amsmath}
\usepackage{amsfonts}
\usepackage{amsthm}
\usepackage{mathrsfs}
\usepackage{amssymb}
\usepackage{csquotes}
\usepackage{graphicx}
\usepackage{enumerate}
\usepackage[mathscr]{euscript}
\DeclareSymbolFont{rsfs}{U}{rsfs}{m}{n}
\DeclareSymbolFontAlphabet{\mathscrsfs}{rsfs}

\theoremstyle{definition}
\newtheorem{Def}{Definition}[section]

\newtheorem{Rmk}[Def]{Remark}

\theoremstyle{plain}
\newtheorem{Prop}[Def]{Proposition}
\newtheorem{Thm}[Def]{Theorem}
\newtheorem{Lemma}[Def]{Lemma}

\newtheorem{Cond}[Def]{Condition}

\newcommand{\R}{\mathbb{R}}

\newcommand{\Z}{\mathbb{Z}}
\newcommand{\N}{\mathbb{N}}

\renewcommand{\epsilon}{\varepsilon}

\title{On the pointwise regularity of the Multifractional Brownian Motion and some extensions}
\author{C. Esser\footnote{ Universit\'e de Li\`ege, D\'epartement de math\'ematique -- zone Polytech 1, 12 all\'ee de la D\'ecouverte, B\^at. B37, B-4000 Li\`ege. celine.esser@uliege.be} and L. Loosveldt\footnote{Universit\'e du Luxembourg, UR en Math\'ematiques, Maison du nombre, 6 avenue de la Fonte, L-4364, Esch-sur-Alzette, Luxembourg. laurent.loosveldt@uni.lu.}}

\begin{document}

\maketitle

\begin{abstract}
We study the pointwise regularity of the Multifractional Brownian Motion and in particular, we get the existence of slow points. It shows that a non self-similar process can still enjoy this property. We also consider various extensions of our results in the aim of requesting a weaker regularity assumption for the Hurst function without altering the regularity of the process.
\end{abstract}


\section*{Introduction}

Given a compact set $[a,b]$ of $(0,1)$ and a function $H \, : \, \R \to [a,b]$, the Multifractional Brownian Motion (MBM) is the process defined in \cite{MR1462329} by the harmonizable representation
\begin{equation} \label{eqn:defmbm}
B_{H}(t) = \int_{\R} \frac{e^{i t \xi}-1}{|\xi|^{H(t)+\frac{1}{2}}} \, d\widehat{W}(\xi),
\end{equation}
where $d\widehat{W}$ is the ``Fourier transform'' of the real-valued white noise measure $dW$. The MBM has also been defined alternatively and independently in  \cite{plv95} by the moving average representation
\begin{equation} \label{eqn:defmbm2}
B_{H}'(t) = \int_{\R} \left( |t-s|^{H(t)-\frac{1}{2}}-|s|^{H(t)-1/2} \right) \, d{W}(s).
\end{equation}
When $H( \cdot)=h$ is a constant function, we recover the  equivalent definitions of the Fractional Brownian Motion (FBM) of Hurst parameter $h$. For this reason,  $H$ is usually called the Hurst function.
Note that the fundamental equality
\begin{align*}
\int_\R f(s) \, dW(s)=\int_\R \widehat{f}(\xi) \, d\widehat{W}(\xi),
\end{align*}
which holds almost surely for all function $f \in L^2(\R)$, insures that, up to a multiplicative deterministic smooth, bounded and non-vanishing function, the processes \eqref{eqn:defmbm} and \eqref{eqn:defmbm2} are identical, see \cite{MR1726364}. 

Generally, one requires that the function $H$ is $\beta$-H\"olderian, for a $\beta>b$. With this assumption, one can recover some of the most fundamental properties of FBM, namely
\begin{enumerate}[(a)]
\item \textbf{Local asymptotic self-similarity: }\cite{MR1462329,MR1922445,MR1967698} for all $t \in \R$,
\[ \lim_{\rho \to 0^+} \text{Law} \left\{ \frac{B_H(t+\rho s)-B_H(t)}{\rho^{H(t)}}, \, s \in \R \right\} = \text{Law} \{B_{H(t)}(s), \,s \in \R \} ,\]
where $\{B_{H(t)}(s), \,s \in \R \}$ is FBM with constant Hurst parameter $H(t)$. In particular, if the function $H$ is non constant, the process $\{ B_{H}(t), \, t \in \R\}$ is not self-similar.

\item \textbf{Uniform modulus of continuity: }\cite{MR1462329} on an event of probability 1, for every open bounded subset $D$ of $\R$,  one has
\[ \limsup_{s,t \in D, |s-t| \to 0^+} \frac{|B_H(s)-B_H(t)|}{|s-t|^{\underline{H}_D} \sqrt{\log|s-t|^{-1}}} = \sqrt{2} C_D\]
 where  $\underline{H}_D = \inf_{t \in D} H(t)$ and $C_D = \sup_{t \in H^{-1}(\underline{H}_D) \cap \overline{D}} C(t)$ with  
\[ C(t) = \sqrt{ \int_{\R} \frac{1-\cos^2(xt)}{|x|^{1+2H(t)}} \, dx}, \quad \forall t \in \R.\]

\item \textbf{Law of iterated logarithm: }\cite{MR1462329} on an event of probability 1, for all $t \in \R$,
\[ \limsup_{s \to t} \frac{|B_H(s)-B_H(t)|}{|s-t|^{H(t)} \sqrt{\log \log|s-t|^{-1}}} = \sqrt{2}C(t) \]

\item \textbf{Existence of local-time}, see e.g. \cite{MR2219711,MR2348754}.
\end{enumerate}
We also refer to the book \cite{MR3839281} for a very complete review on the subject.

In the present paper, we will be particularly interested in the pointwise regularity of MBM. Using well-established terminologies by Kahane \cite{MR833073}, (c) here over means that almost surely, almost every point $t \in \R$ is ordinary while by (b), we know that, for some points called rapid, the oscillation is faster. Moreover, concerning the Brownian Motion (BM), Kahane has pointed out in \cite{MR833073} the existence of a third family of points, presenting a slower oscillation. Recently, we have extended this fact to FBM and moreover, we have shown that this property is somehow exceptional using two different notions of genericity, see  \cite{esserloosveldt}.  A natural question is therefore to understand where does this particularity come from. The extension from BM to FBM has underlined that the existence of slow points does not depend on the Markovian property of the process. The results obtained in \cite{dawloosveldt} have highlighted the fact that it does not depend on the Gaussianity of the process neither: indeed, the (generalized) Rosenblatt process, which is known to be non-Gaussian, does present slow points in its pointwise regularity. Here, we show that this does not depend on the self-similarity neither, as MBM displays slow points.


The paper is organized as follows.
Section 1 is devoted to the proof of upper bounds for the regularity of the MBM, while the optimality of these bounds is studied in Section 2. Note that these sections rely deeply on a wavelet-type expansion of MBM first given in \cite{MR2177638}. In Section 3,   by slightly modifying  this expansion,  we show that one can relax some hypothesis made on the function $H$ without altering the pointwise regularity properties of the process. This model could be of interest for simulation purposes, when we have to consider multifractional phenomena with few regularity assumptions for the Hurst function.

\section{A sharp upper bound for some oscillations}

The definition \eqref{eqn:defmbm} of MBM naturally leads  to consider the following Gaussian field.
\begin{Def}
The generator of Multifractional Brownian motion (gMBM) is the Gaussian field $\{ B(t,\theta) \, : \, (t, \theta) \in \R \times (0,1) \}$ defined as
\[ B(t,\theta) = \int_{\R} \frac{e^{i t \xi}-1}{|\xi|^{\theta+\frac{1}{2}}} \, d\widehat{W}(\xi).\]
\end{Def}
From \eqref{eqn:defmbm}, it is clear that, for all $t \in \R$, we have $B_H(t)=B(t,H(t))$.

A wavelet-type expansion of gMBM is the crucial point for our analysis of the pointwise regularity of MBM. Let us introduce it in a few words, details can be found in the paper \cite{MR2177638}. In what follows, $\{2^{j/2} \psi(2^j \cdot -k)  \, : \, (j,k) \in \Z^2 \}$ stands for the Lemari\'e-Meyer orthonormal wavelet base of the Hilbert space $L^2(\R)$, introduced in \cite{MR864650}. Its particular features are the fact that the mother wavelet $\psi$ belongs to the Schwartz class of $C^\infty$ functions whose derivatives of any order have fast decay and that $\widehat{\psi}$ is compactly supported and is vanishing in a neighbourhood of $0$. Thanks to these facts, the function
\[ \Psi \, : \, (t, \theta) \in \R^2 \mapsto \int_{\R} \frac{e^{it\xi} \widehat{\psi}(\xi)}{|\xi|^{\theta + \frac{1}{2}}} \, d\xi \]
is well-defined. Moreover, one can check that $\Psi$ belongs to $C^\infty(\R^2)$ and satisfies the following fast decay property: for all $a,b \in \R$ and $L,m,n \in \N$, 
\begin{equation}\label{eqn:fastdecay}
\sup_{\theta \in [a,b]} \sup_{t \in \R} (3+|t|)^L |D^m_t D^n_\theta \Psi(t,\theta)| < \infty,
\end{equation}
see \cite[Lemma 2.1]{MR2177638} for a proof of this fact. This function leads to the following expansion for gMBM, for all $\R \times (0,1)$, we have
\begin{equation}\label{eqn:wavexp}
B(t,\theta) = \sum_{j \in \Z} \sum_{k \in \Z} 2^{-j \theta} \varepsilon_{j,k}  \left( \Psi(2^j t-k, \theta)-\Psi(-k, \theta) \right),
\end{equation}
where $(\varepsilon_{j,k})_{(j,k) \in \Z^2}$ is a sequence of i.i.d. $\mathcal{N}(0,1)$ random variables. The convergence of this series holds  in $L^2(\Omega)$, as a consequence of Wiener isometry, but also, most importantly in our case, almost surely uniformly on every compact subset of $\R \times (0,1)$. Among other things, this fact is a consequence of the following important estimate that can be done on the family $(\varepsilon_{j,k})_{(j,k) \in \Z^2}$.

\begin{Lemma}\label{boundnormal} \cite{MR2027888}
Let $(\varepsilon_{j,k})_{(j,k) \in \Z^2}$ be a sequence of independent $\mathcal{N}(0,1)$ random variables. There are an event $\Omega^*_0$ of probability $1$ and a positive random variable $C_1$ of finite moment of every order such that, for all $\omega \in \Omega^*_0$ and $(j,k) \in \Z^2$, the inequality
\begin{align}\label{eq:boundnorm}
|\varepsilon_{j,k}(\omega)| \leq C_1(\omega) \sqrt{\log(3+j+|k|)}
\end{align}
holds.
\end{Lemma}

This last Lemma is also useful to show that the sample paths of the field
\[ \overline{B}(t,\theta):= \sum_{j=- \infty}^{-1} \sum_{k \in \Z} 2^{-j \theta} \varepsilon_{j,k}  \left( \Psi(2^j t-k, \theta)-\Psi(-k, \theta) \right) \]
are almost surely $C^\infty$ functions.

In the sequel, we assume the following condition for the Hurst function. It is slightly less restrictive that the original uniform H\"older regularity assumption required in \cite{MR1462329,plv95}. Note that it is the condition used in \cite[Theorem 1.89]{MR3839281} to study the pointwise H\"older exponent of MBM.

\begin{Cond}\label{condi}
The Hurst function $H \, : \, \R \to [a,b]$, with $0<a<b<1$, is such that for all $t \in \R$, there exists $\gamma \geq H(t)$ such that $H$ belongs to the pointwise H\"older space $C^\gamma (t)$, which means that there exist $R_t>0$ and $c_t>0$ such that
\[ |H(s)-H(t)| \leq c_t |s-t|^{\gamma}\]
for all $s \in \R$ with $|s-t| \leq R_t$.
\end{Cond}

Our main result of this section can be stated as follows.

\begin{Thm}\label{thm:slow}
If the function $H \, : \, \R \to [a,b]$ satisfies the Condition \ref{condi} then almost surely, for every non-empty interval $I$ of $\R$, there exists $t \in I$ such that
\begin{equation}\label{eqn:slow}
\limsup_{s \to t} \frac{|B_H(s)-B_H(t)|}{|s-t|^{H(t)}} < \infty.
\end{equation}
\end{Thm}

The proof of Theorem \ref{thm:slow} uses the wavelet series representation \eqref{eqn:wavexp} that gives for all $t\in \R$
\begin{equation}\label{eqn:waveletdecomp}
B_H(t) = \sum_{j \in \Z} \sum_{k \in \Z} 2^{-j H(t)} \varepsilon_{j,k}  \left( \Psi(2^j t-k,  H(t))-\Psi(-k,  H(t)) \right).
\end{equation}
First, note that Condition \ref{condi} and the fact that the trajectories of the field $\overline{B}$ are almost surely $C^\infty$ functions  entail that almost surely \begin{equation}\label{eqn:reglowfrequency}
\limsup_{s \to t} \frac{|\overline{B}(s,H(s))-\overline{B}(t,H(t))|}{|s-t|^{H(t)}} < \infty
\end{equation}
for all $t\in \R$.
Therefore, we only need to analyse the high frequency part of MBM, which can be done through the field
\begin{equation}\label{eqn:highfrequency}
\widetilde{B}(t,\theta):=\sum_{j \in \N} \sum_{k \in \Z} 2^{-j \theta} \varepsilon_{j,k}  \left( \Psi(2^j t-k, \theta)-\Psi(-k, \theta) \right).
\end{equation}

\begin{Rmk}\label{rmk:reduc01}
Before going further, let us remark that one can reduce our work to
the proof of the existence of a point satisfying \eqref{eqn:slow} in the interval $[0,1)$. Indeed, let us recall that any open interval in $\R$ can be written as a
countable union of dyadic intervals $(
\lambda_{j,k}=[k2^{-j},(k+1)2^{-j}) )_{j \in \N, k \in
  \Z}$. Therefore, in order to prove Theorem \ref{thm:slow}, it is
sufficient to show that, for all $j \in \N$ and $k \in \Z$, there
exists an event $\Omega_{j,k}$ of probability $1$ such that, for all
$\omega \in \Omega_{j,k}$, there exists $t \in \lambda_{j,k}$ which
satisfies \eqref{eqn:slow}. Now, up to dilatations and translations, it suffices to consider the dyadic interval $\lambda_{0,0}=[0,1)$.
\end{Rmk}

Let us come back to the field \eqref{eqn:highfrequency}. This last one has been largely considered in \cite{MR2743002}, where an alternative wavelet-type expansion of MBM is given. Let us already mention that we will also be interested in this representation in the last section of this paper. From now, as in \cite{MR2743002}, for all $j \in \N$ and $k \in \Z$, the notation $g_{j,k}$  stands for the function
\[ g_{j,k} \, : \, (t,\theta) \mapsto 2^{-j \theta}  \left( \Psi(2^j t-k, \theta)-\Psi(-k, \theta) \right). \]

\begin{Lemma}\label{lem:fromAB}
For all compact interval $K$ of $[0,1)$ and for all $n \in \N$, there exists a deterministic  constant $c_{K,n}>0$ such that, for all $\omega \in \Omega^*$
\[ \sup_{t \in K, \theta \in [a,b]} \left( \sum_{j \in \N} \sum_{k \in \Z} |D_\theta^n g_{j,k}(t,\theta)| |\varepsilon_{j,k}(\omega)|\right) \leq c_{K,n}C_1(\omega).  \]
In particular, for every $t \in K$ and every $\theta_1, \theta_2 \in [a,b]$, we have
$$
|B(t, \theta_1) - B(t, \theta_2)| \leq c_{K,1}C_1(\omega) |\theta_1 - \theta_2|.
$$
\end{Lemma}

\begin{proof}
The first part is a direct consequence of Lemma \ref{boundnormal} and \cite[Lemma 2]{MR2743002}. The second part is obtained by applying  the Mean Value Theorem to each function $g_{j,k}(t, \cdot)$.
\end{proof}

The strategy to prove the existence of slow points relies on a procedure which allows to deduce a sharper estimate than the inequality \eqref{eq:boundnorm} obtained in Lemma \ref{boundnormal} for the random variables. This procedure was initiated by Kahane for BM \cite{MR833073} and we have generalized it for FBM \cite{esserloosveldt}. This generalized version can also be applied in the present setting and can be summarized in the following Theorem. All along this paper, given $t \in \R$ and $j \in \N$, $k_j(t)$ stands for the unique integer such that $t \in [k_j(t)2^{-j},(k_j(t)+1)2^{-j})$.

\begin{Thm}\cite{esserloosveldt}\label{thm:procedure}
Let us fix $m>0$. There exists an event $\Omega^* _1$ such that for every $\omega \in \Omega^* _1$, there are $\mu>0$ and $t \in (0,1)$ such that 
\begin{equation}\label{eqn:slownesti}
|\varepsilon_{j,k} (\omega)| \leq 2^l \mu
\end{equation} for every $j \in \N_0$ and every $k \in \Lambda^{l}_{j,m} (t) $, 
where $$\Lambda_{j,m}^0(t) = \{ k \in \Z : \, |k_j(t)-k| \leq 1 \}$$ and for all $l \geq 1$, $$
\Lambda_{j,m}^l(t) = \{ k \in \Z : \, 2^{m(l-1)}<|k_j(t)-k| \leq 2^{ml} \}.$$ The set of such points is denoted $S_{\text{low},m}^\mu$.
\end{Thm}

From now on and until the end of this section, we fix $m >0$ such that $\frac{1}{m} <a$ and denote 
\begin{equation}\label{eqn:Omegastar}
\Omega^* = \Omega^*_0 \cap \Omega^*_1
\end{equation}
 the event of probability $1$ obtained as the intersection of the events of probability 1 given by Lemma \ref{boundnormal} and Theorem \ref{thm:procedure} respectively. 
For all $j \in \N$, let us set
$$
\widetilde{B}_j(t,\theta):=\sum_{k \in \Z} 2^{-j \theta} \varepsilon_{j,k}  \left( \Psi(2^j t-k, \theta)-\Psi(-k, \theta) \right).
$$
Note that on $\Omega^\ast$, since inequality \eqref{eq:boundnorm} holds, it is straightforward to check that  the trajectories of the field $\widetilde{B}_j$ are continuously differentiable, using the fast decay property \eqref{eqn:fastdecay}.

\begin{Lemma}\label{lemma1}
On the event $\Omega^*$ of probability $1$, there exists a deterministic constant $c_m>0$ such that, for all $n \in \N$ and $\mu>0$, if $t \in S_{\text{low,m}}^\mu $ and $\varepsilon >0$ is such that $I_\varepsilon(t):=[t-\varepsilon,t+ \varepsilon] \subset (0,1)$, then
\[ \left|\sum_{j=0}^n (\widetilde{B}_j(t,\theta_1)-\widetilde{B}_j(s,\theta_2))\right| \leq c_m \mu  2^{-\theta_1 n} 2^{n|\theta_1-\theta_2|} + C^* \, c_{I_\varepsilon(t),1}  |\theta_1-\theta_2|\]
for all $s \in I_\varepsilon(t)$ with $|s-t| \leq 2^{-n+1}$ and $\theta_1, \theta_2 \in [a,b]$.
\end{Lemma}

\begin{proof}
If $t \in S_{\text{low,m}}^\mu $, $s \in [t-\varepsilon,t+ \varepsilon]$ with $|s-t| \leq 2^{-n+1}$ and $\theta_1, \theta_2 \in [a,b]$, then by the Taylor formula at first order, there exist $x$ between $s$ and $t$, and $\xi$ between $\theta_1$ and $\theta_2$ such that
\begin{align}\label{eqn:aftertaylor}
\sum_{j=0}^n (\widetilde{B}_j(t,\theta_1)-\widetilde{B}_j(s,\theta_2)) & =(t-s) \sum_{j=0}^n \sum_{k \in \Z} \varepsilon_{j,k} D_t g_{j,k}(x,\xi) \nonumber \\ & \quad + (\theta_1-\theta_2) \sum_{j=0}^n \sum_{k \in \Z} \varepsilon_{j,k} D_\theta g_{j,k}(x,\xi).
\end{align}
The second series in the right-hand side of the equality \eqref{eqn:aftertaylor} is bounded by Lemma \ref{lem:fromAB}. In order to control the first series in the right-hand side of \eqref{eqn:aftertaylor}, as $D_t g_{j,k}(x,\xi)=2^{j(1-\xi)} D_t\Psi(2^jx-k,\xi)$, we use the fast decay property \eqref{eqn:fastdecay}  to get, for all $0 \leq j \leq n$,
\begin{align}\label{decomposomme}
\sum_{k \in \Z}\left| \varepsilon_{j,k} D_t g_{j,k}(x,\xi) \right| & \leq c_1 2^{j(1-\xi)} \sum_{l \in \N} \sum_{k \in \Lambda_{j,m}^l(t)} |\varepsilon_{j,k}| \frac{1}{(3+|2^{j}x -k|)^4},
\end{align}
for a deterministic positive constant $c_1$. Now, note that for all $l \geq 1$ and $k \in \Lambda_{j,m}^l(t)$, we have
\begin{align*}
|2^jx-k| & \geq |k_j(t)-k|-|2^jx-k_j(t)| \\
&\geq |k_j(t)-k|-(|2^jx-2^j t| + |k_j(t)-2^j t|) \\ 
&\geq 2^{m(l-1)}-3,
\end{align*}
because $j \leq n$. Together with \eqref{decomposomme} and inequality \eqref{eqn:slownesti}, it implies that 
\begin{align*}
\sum_{k \in \Z}\left| \varepsilon_{j,k} D_t g_{j,k}(x,\xi) \right|  & \leq c_1 2^{j(1-\xi)} \sum_{l \in \N} \sum_{k \in \Lambda_{j,m}^l(t)} 2^l \mu \frac{1}{(3+|2^{j}x -k|)^4} \\
& \leq c_2 2^m \mu 2^{j(1-\xi)}\sum_{k \in \Z} \frac{1}{(3+|2^{j}x -k|)^3} \\
& \leq c_3 \mu 2^{j(1-\xi)} 
\end{align*}
where $c_2$ and $c_3$ are positive deterministic constants only depending on $m$. Thus,
\begin{align*}
\left|\sum_{j=0}^n \sum_{k \in \Z} \varepsilon_{j,k} D_t g_{j,k}(x,\xi) \right|& \leq c_3 \mu \sum_{j=0}^n 2^{j(1-\xi)}\leq c_4 2^{n(1-\xi)}\leq c_4 2^{-\theta_1 n} 2^{n|\theta_1-\theta_2|}
\end{align*}
for a deterministic constant $c_4$ only depending on $m$, since $\xi \in(0,1)$ is between $\theta_1$ and $\theta_2$. 
\end{proof}

\begin{Lemma}\label{lemma2}
On the event $\Omega^*$ of probability $1$, there exists a deterministic constant $c_m>0$ such that, for all  $\mu>0$, $\theta\in [a,b]$ and $t \in S_{\text{low,m}}^\mu  $
\begin{enumerate}
\item one has  \[ \left| \sum_{k \in \N} \varepsilon_{j,k} \Psi(2^j t-k,\theta) \right| \leq c_m \mu ,\]
\item if $\varepsilon >0$ is such that $I_\varepsilon(t):=[t-\varepsilon,t+ \varepsilon] \subset (0,1)$, then for all  $n \in \N$, $s \in I_\varepsilon(t)$ with $ |s-t| \leq 2^{-n+1}$ and $j \geq n$,  one has
\[ \left| \sum_{k \in \N} \varepsilon_{j,k} \Psi(2^j s-k,\theta) \right|  \leq c_m \mu 2^{\frac{1}{m}(j-n)}.\]
\end{enumerate}
\end{Lemma}

\begin{proof}
The first bound is obtained exactly as in \eqref{decomposomme}, partitioning the sum over $k \in \Z$ with the subsets $\Lambda_{j,m}^l(t)$ and using the fast decay property \eqref{eqn:fastdecay} for $\Psi$. Concerning the second bound, we note that, if $l$ is the greatest integer for which $|s-t| \geq 2^{ml} 2^{-j}$ then, for all $l' \in \N$ and $k \in \Lambda_{j,m}^{l'}(s)$, the construction gives
\[ |\varepsilon_{j,k}| \leq 2^{l'+l} \mu.\]
As $|s-t| \leq 2^{-n+1}$, we deduce $l \leq \frac{1}{m}(j+1-n)$ and we obtain the desired upper bound by partitioning the sum over $k \in \Z$ using the subsets $\Lambda_{j,m}^{l'}(s)$.
\end{proof}

Let us recall that, for all $L>1$, there exists a deterministic constant $c>0$ such that, for all $j \in \Z$ and $x \in \R$,
\begin{equation}\label{bound:log}
\sum_{k \in \Z} \frac{\sqrt{\log(3+|j|+|k|}}{(3+|2^jx-k|)^L}  \leq c \sqrt{\log(3+|j|+2^j|x|)},
\end{equation}
see for instance \cite[ Lemma 4.2]{MR4110623} for a proof.

\begin{Lemma} \label{lemma3}
On the event $\Omega^*_0$  of probability $1$, there exists a deterministic constant $c_1>0$ such that, for all $\theta_1, \theta_2 \in [a,b]$ and $j \in \N$,
\[ \left| \sum_{k \in \N} \varepsilon_{j,k} \left( 2^{-j \theta_1} \Psi(-k,\theta_1)-2^{-j \theta_2} \Psi(-k,\theta_2) \right) \right| \leq C^* \, c_1|\theta_1-\theta_2| 2^{-j a } \sqrt{\log(3+j)}.\]
\end{Lemma}

\begin{proof}
If inequality \eqref{eq:boundnorm} holds, we know form the fast decay property \eqref{eqn:fastdecay}, that the function
\[ \theta \mapsto \sum_{k \in \N} \varepsilon_{j,k} 2^{-j \theta} \Psi(-k,\theta) \]
is smooth. Therefore, using the Taylor formula at first order, we obtain the existence of $\xi$ between $\theta_1$ and $\theta_2$ such that
\begin{align*}
 \sum_{k \in \N} \varepsilon_{j,k} \left( 2^{-j \theta_1} \Psi(-k,\theta_1)-2^{-j \theta_2} \Psi(-k,\theta_2) \right)=(\theta_1- \theta_2) \sum_{k \in \N} \varepsilon_{j,k} 2^{-j \xi} \Psi(-k,\xi).
\end{align*}
Using \eqref{eq:boundnorm}, the fast decay property \eqref{eqn:fastdecay} and inequality \eqref{bound:log}, we get
\begin{align*}
\left|\sum_{k \in \N} \varepsilon_{j,k} 2^{-j \xi} \Psi(-k,\xi) \right| & \leq C^* 2^{-j \xi} \sum_{k \in \Z} \frac{\sqrt{\log(3+j+|k|)}}{(3+|k|)^4}
& \leq C^* c_1 2^{-j \xi} \sqrt{\log(3+j)} \\
& \leq  C^* c_1 2^{-j a } \sqrt{\log(3+j)},
\end{align*}
for a positive deterministic constant $c_1$. The conclusion follows immediately.
\end{proof}

We have now enough material to prove Theorem \ref{thm:slow}.

\begin{proof}[Proof of Theorem \ref{thm:slow}]
In view of \eqref{eqn:reglowfrequency} and Remark \ref{rmk:reduc01}, it suffices to show that on the event $\Omega^*$ of probability $1$ defined in \eqref{eqn:Omegastar} there exists $t \in (0,1)$ such that
$$
\limsup_{s \to t} \frac{|\widetilde{B}(s,H(s))-\widetilde{B}(t,H(t))|}{|s-t|^{H(t)}} < \infty.
$$
Let us recall that we have fixed $m \in \N$ such that $\frac{1}{m}<a$. Theorem \ref{thm:procedure} allows to consider  $t \in S_{\text{low,m}}^\mu$ for some $\mu >0$. Now, if  $s \in (0,1)$ is such that $2^{-n} \leq |s-t| \leq 2^{-n+1}$, we write
\begin{align*}
\left|\widetilde{B}(t,H(t))-\widetilde{B}(s,H(s))\right| & \leq   \left|\sum_{j=0}^n ( \widetilde{B}_j(t,H(t))-\widetilde{B}_j(s,H(s))) \right| \\
& + \sum_{j \geq n+1}  \left| \sum_{k \in \N} \varepsilon_{j,k} 2^{-j H(t)} \Psi(2^j t-k,H(t)) \right| \\
& + \sum_{j \geq n+1}  \left| \sum_{k \in \N} \varepsilon_{j,k} 2^{-j H(s)} \Psi(2^j s-k,H(s)) \right| \\
& + \sum_{j \geq n+1} \left| \sum_{k \in \N} \varepsilon_{j,k} \left( 2^{-j H(t)} \Psi(-k,H(t))-2^{-j H(s)} \Psi(-k,H(s)) \right) \right|.
\end{align*}

From Condition \ref{condi} we know that, if $n$ is large enough,
\begin{equation}\label{eqn:useful}
|H(t)-H(s)| \leq c_t|t-s|^{H(t)}.
\end{equation}
Thus, Lemma \ref{lemma1} and \eqref{eqn:useful} combined with the fact that $2^{-n} \leq |s-t| \leq 2^{-n+1}$ give 
\begin{align*}
\left|\sum_{j=0}^n ( \widetilde{B}_j(t,H(t))-\widetilde{B}_j(s,H(s))) \right| &\leq c_m \mu |t-s|^{H(t)} 2^{ c_t n 2^{-(n-1)H(t)}} + C^* c_{I_\varepsilon(t),1} c_t|t-s|^{H(t)} \\
& \leq (c_m \mu+ C^* c_{I_\varepsilon(t),1}) |t-s|^{H(t)}. 
\end{align*}
Using $2^{-n} \leq |s-t| \leq 2^{-n+1} $, Lemma \ref{lemma2} and \eqref{eqn:useful} imply  
\[ \sum_{j \geq n+1}  \left| \sum_{k \in \N} \varepsilon_{j,k} 2^{-j H(t)} \Psi(2^j t-k,H(t)) \right| \leq 2 c_m \mu 2^{-n H(t)}\leq 2 c_m \mu |s-t|^{H(t)}\]
while, as $\frac{1}{m} < a $,
\begin{align*}
 \sum_{j \geq n+1}  \left| \sum_{k \in \N} \varepsilon_{j,k} 2^{-j H(s)} \Psi(2^j s-k,H(s)) \right| & \leq c_m \mu \sum_{j \geq n+1} 2^{(\frac{1}{m}-H(s))(j-n)} 2^{-H(s) n} \\
 &  \leq c_m \mu \sum_{j \geq n+1} 2^{(\frac{1}{m}-a)(j-n)} 2^{-H(s) n} \\
 & \leq  c_m \mu |t-s|^{H(t)} 2^{|H(t)-H(s)| n} \\
 & \leq  c_m \mu |t-s|^{H(t)}.
\end{align*}
Finally, by Lemma \ref{lemma3} and \eqref{eqn:useful}
\[ \sum_{j \geq n+1} \left| \sum_{k \in \N} \varepsilon_{j,k} \left( 2^{-j H(t)} \Psi(-k,H(t))-2^{-j H(s)} \Psi(-k,H(s)) \right) \right| \leq C^* \, c_2 |t-s|^{H(t)}, \]
for a deterministic positive constant $c_2$.
\end{proof}

\begin{Rmk}\label{rmk:rapord}
Theorem \ref{thm:slow} implies that if $t$ is a  point which satisfies
\eqref{eqn:slow}, then  $r \mapsto |r|^{H(t)}$ is a pointwise modulus  of continuity for $B_H$ at $t$.
Let us remark that our strategy can also be applied to recover the upper bounds for the well-known uniform modulus of continuity as well as the law of iterated logarithm. Let us explain how to adapt our proofs on this purpose.

Concerning the uniform modulus of continuity, if $s,t \in [0,1]$ are such that  $2^{-n} \leq |s-t| \leq 2^{-n+1}$ and $\theta_1$ and $\theta_2$ are fixed in $[a,b]$, we know that, almost surely, one can write \eqref{eqn:aftertaylor}. Therefore, if inequality \eqref{eq:boundnorm} holds, using  $D_t g_{j,k}(x,\xi)=2^{j(1-\xi)} D_t\Psi(2^jx-k,\xi)$, the fast decay property \eqref{eqn:fastdecay} and \eqref{bound:log}, one has
\begin{align}\label{eqn:rapid1}
\left| \sum_{j=0}^n \sum_{k \in \Z} \varepsilon_{j,k} D_t g_{j,k}(x,\xi) \right| & \leq  c \, C_1  \sum_{j=0}^n \sum_{k \in \Z} 2^{j(1-\xi)}\frac{\sqrt{\log(3+|j|+|k|)}}{(3+|2^jx-k|)^L} \nonumber\\
& \leq c\,  C_1 \sum_{j=0}^n 2^{j(1-\xi)} \sqrt{\log(3+|j|+2^j|x|)} \nonumber \\
& \leq c\, C_1 \sum_{j=0}^n 2^{j(1-\xi)} \sqrt{j} \nonumber \\
& \leq c \,C_1 2^{n(1-\xi)} \sqrt{n}
\end{align}
where $c$ is a positive deterministic constant which value may differ from a line to another but does not depend on any relevant quantities. Similarly, for all $j>n$, we have
\begin{equation}\label{eqn:rapid2}
\left| \sum_{k \in \N} \varepsilon_{j,k} \Psi(2^j t-k,\theta_1) \right| \leq c \, C_1 \sqrt{j}
\end{equation}
and 
\begin{equation}\label{eqn:rapid3}
\left| \sum_{k \in \N} \varepsilon_{j,k} \Psi(2^j s-k,\theta_2) \right| \leq c \, C_1  \sqrt{j}.
\end{equation}
Therefore, gathering the expression \eqref{eqn:aftertaylor}, Lemma \ref{lem:fromAB}, the inequalities \eqref{eqn:rapid1}, \eqref{eqn:rapid2} and \eqref{eqn:rapid3}, Lemma \ref{lemma3} and Condition \ref{condi}, we get
\begin{align*}
|B_H(s)-B_H(t)| & \leq c \, C_1 \Bigg( (|t-s| 2^{n(1-\xi)} \sqrt{n} + |H(t)-H(s)|  \\
& + \sum_{j>n}  \left( 2^{-j H(t)} \sqrt{j} + 2^{-j H(s)} \sqrt{j} + |H(t)-H(s)| 2^{-ja} \sqrt{\log(3+j)} \right) \Bigg) \\
& \leq c \, C_1 \left( |t-s| 2^{n(1-\xi)} \sqrt{n} + 2^{-n H(t)} \sqrt{n}+  2^{-n H(s)} \sqrt{n} + |H(t)-H(s)| \right) \\
& \leq c \, C_1 \left( 2^{-n H(t)} \sqrt{n}  + |t-s|^{H(t)} \right)\\
& \leq |t-s|^{H(t)} \sqrt{\log|s-t|^{-1}}.
\end{align*}
Thus, we have shown that, almost surely, for all $t \in [0,1]$
\begin{equation}\label{eqn:limrap}
 \limsup_{s \to t} \frac{|B_H(s)-B_H(t)|}{|t-s|^{H(t)} \sqrt{\log|s-t|^{-1}}} < \infty.
\end{equation}

Concerning the law of iterated logarithm, by an indexing argument, one can note that for all $t \in [0,1]$, there exits an event $\Omega_t^*$ of probability $1$ and a positive random variable $c_t$ of finite moment of any order such that, for all $\omega \in \Omega_t^*$,
\begin{equation}\label{eqn:moustiquet}
|\varepsilon_{j,k}(\omega)| \leq c_t \sqrt{\log(3+|j|+|k-k_j(t)|)}.
\end{equation}
Then we use \cite[Lemma 3.22]{dawloosveldt} which gives, for all $L$, for all $n \in \N$,
\begin{equation}\label{eqn:ineforord1}
 \sum_{k \in \Z} \frac{\sqrt{\log(3+j+|k-k_j(t)|}}{(3+|2^{j}x-k|)^{L}} \leq c \sqrt{\log(3+j)}
\end{equation}
if $0\leq j \leq n $, and 
\begin{equation}\label{eqn:ineforord2}
\sum_{k \in \Z} \frac{\sqrt{\log(3+j+|k-k_j(t)|}}{(3+|2^{j}x-k|)^{L}} \leq c \sqrt{j-n+1}\sqrt{\log(3+j)}.
\end{equation}
for all $j > n $.
Adapting what has been done for \eqref{eqn:rapid1}, \eqref{eqn:rapid2} and \eqref{eqn:rapid3} by using inequality \eqref{eqn:moustiquet} instead of \eqref{eq:boundnorm} and inequalities \eqref{eqn:ineforord1} and \eqref{eqn:ineforord2} instead of \eqref{bound:log}, we get, for all $t,s \in [0,1]$,
\[ |B_H(s)-B_H(t)| \leq c c_t |t-s|^{H(t)} \sqrt{\log\log|s-t|^{-1}}. \]
In particular, by Fubini theorem, almost surely, for almost every $t \in [0,1]$,
\begin{equation}\label{eqn:limord}
 \limsup_{s \to t} \frac{|B_H(s)-B_H(t)|}{|t-s|^{H(t)} \sqrt{\log\log|s-t|^{-1}}} < \infty.
\end{equation}

\end{Rmk}

\section{Optimality of the upper bound}

Let us now focus on the optimality of the previously obtained modulus
of continuity. It is worth mentionning that the optimatity of the pointwise modulus of
continuity $r \mapsto r^{H(t)}$ for some points  $t$ of
the MBM  given in Theorem \ref{thm:slow} has been already obtained in
\cite[Theorem 6.17]{MR3839281}, under the following assumption on the
Hurst function $H$, which is a little bit stronger than Condition
\ref{condi}.

\begin{Cond}\label{condii}
The Hurst function $H \, : \, \R \to [a,b]$, with $0<a<b<1$, is such that for all $t \in \R$, there exists $\gamma > H(t)$ such that $H$ belongs to the pointwise H\"older space $C^\gamma (t)$.
\end{Cond}

We  show  in this section how to extend the
proof of \cite[Theorem 6.17]{MR3839281} to get the optimatily of the two other pointwise
modulus of continuity $r \mapsto r^{H(t)} \sqrt{\log r^{-1}} $ and $r \mapsto
r^{H(t) }\sqrt{\log\log r^{-1}}$.  Using the terminology of
\cite{esserloosveldt} inspired by the work of Kahane \cite{MR833073},
the points $t$ for which the modulus of continuity given by
\eqref{eqn:slow} is optimal are called slow points. Similarly, those for which
\eqref{eqn:limrap} is optimal are called fast points. Finally, we will get the existence of the so-called ordinary points: for
almost every point $t$, the modulus of continuity given in
\eqref{eqn:limord} is optimal.

\begin{Thm}\label{thm:optimal}
If the function $H \, : \, \R \to [a,b]$ satisfies the Condition \ref{condii} then almost surely, for every non-empty interval $I$ of $\R$,
\begin{itemize}
\item  there exists $t \in I$ such that
$$
0<\limsup_{s \to t} \frac{|B_{H}(s)-B_{H}(t)|}{|s-t|^{H(t)} \sqrt{\log|s-t|^{-1}}} < \infty.
$$
Such a point 
is called a rapid point.
\item almost every point $t \in I$ is such that
$$
0<\limsup_{s \to t} \frac{|B_{H}(s)-B_{H}(t)|}{|s-t|^{H(t)} \sqrt{\log\log|s-t|^{-1}}} < \infty.
$$
Such a point 
is called an ordinary point.
\item  there exists $t \in I$ such that
$$
0<\limsup_{s \to t} \frac{|B_{H}(s)-B_{H}(t)|}{|s-t|^{H(t)}} < \infty.
$$
Such a point 
is called a slow point.
\end{itemize}
\end{Thm}

The method for the proof of Theorem \ref{thm:optimal} is based on a
wavelet-type argument and relies on a biorthogonality proprerty of
sequences of functions defined via $\Psi$: For any fixed $\theta \in \R$, the two sequences of functions
  $$
  \{ 2^{j/2} \Psi(2^{j}\cdot -k, \theta): (j,k) \in \Z^{2}\}
  \quad \text{and} \quad
  \{2^{j/2} \Psi(2^{j}\cdot -k, -\theta-1): (j,k) \in \Z^{2}\}
  $$
  are biorthogonal in $L^{2}(\R)$, see \cite[Proposition
5.13 (i)]{MR3839281}. This result allows to express
coefficients appearing in the decomposition \eqref{eqn:wavexp} in
terms of $B$, see \cite[Lemma 6.22]{MR3839281}.

\begin{Lemma}\label{lem:wavexpr}\cite{MR3839281}
On the event $\Omega^{\ast}_0$ of probability $1$, for every $\theta \in (0,1)$ and
 for all $(j,k) \in \Z^{2}$, one has
 $$
2^{-j \theta} \varepsilon_{j,k} = 2^{j} \int_{\R} B(u, \theta)
\Psi(2^{j}u-k, -\theta-1) du
$$
where $\varepsilon_{j,k}$ is given by the representation \eqref{eqn:wavexp}.
\end{Lemma}

In the case of the MBM $B_{H} $,   Lemma \ref{lem:wavexpr} allows to write
\begin{equation}\label{eq:wavexpr}
    2^{-j H(t)} \varepsilon_{j,k_{j}(t)} = 2^{j} \int_{\R} B(u, H(t))
\Psi(2^{j}u-k_{j}(t), -H(t)-1) du.
\end{equation}





In order to state the next result, we recall  that a modulus of continuity is an increasing function $\sigma:\R^+ \to  \R^+$ satisfying 
 $\sigma (0) = 0$  and for which there is $C>0$ such that 
 $\sigma (2x) \leq C \sigma (x)  $
 for all $x \in \R^{+}$.  We say that $\sigma$ in submultiplicative if
 $\sigma (xy) \leq \sigma(x) \sigma(y)$ for all $x,y \in \R^{+}$.

 \begin{Rmk}\label{rem:submult}
   It is very classical that
   $$
   \sqrt{1+x+y} \leq \sqrt{1+x}\sqrt{1+y} \quad \text{and} \quad \log(3+x+y) \leq \log (3+x) \log (3+y)
   $$
   for all $x,y \in \R^{+}$. 
      It follows that the moduli of
      continuity of interest in Theorem \ref{thm:optimal} for the
       rapid and ordinary points are
      asymptotically equivalent as $r \to 0^{+}$ to the submultiplicative
      modulus of continuty given respectively by
      $$
      r \mapsto r^{H(t)} \sqrt{1+ \log r^{-1}} \quad \text{and} \quad
      r \mapsto r^{H(t)} \sqrt{1+ \log(3+ \log r^{-1})}.
      $$
 \end{Rmk}
 
 \begin{Rmk}
   The assumption that $\sigma$ in submultiplicative can be slightly
   weakened by imposing the existence of a constant $C>0$ such that
   $\sigma (xy) \leq C \sigma(x) \sigma(y)$ for all $x,y \in \R^{+}$. 
\end{Rmk}

 \begin{Prop}\label{prop:controlecoeff}
Let us consider $t \in \R$ and a submultiplicative modulus of
continuity $\sigma$ with  polynomial growth. 
Assume that  the Hurst function $H \, : \, \R \to [a,b]$, with $0<a<b<1$, is
such that there exists $\gamma > 0$ such that $H$ belongs to the
pointwise H\"older space $C^\gamma (t)$. Then on the event $\Omega^*_0$  of probability $1$, one has for every
$j$ large enough
$$
2^{-j H(t)} |\varepsilon_{j,k_{j}(t)}| \leq C \Big(\sup\left\{ 
\frac{|B_{H}(s) - B_{H} (t)|}{ \sigma(|s-t|)} :  |s-t| < c   2^{-j/2}\right\}
 \sigma(2^{-j}) +
2^{-\gamma j }\Big)
 $$
for a deterministic constant $c>0$ and a positive random variable $C$,
 where the
variables $\varepsilon_{j,k}$ are given by \eqref{eqn:waveletdecomp}
and where the supremum may take the value $+ \infty$. 
\end{Prop}

\begin{proof}
As the first moment of $\Psi$ vanishes, see  \cite[Remark
5.12]{MR3839281},  we get by using the equality \eqref{eq:wavexpr} and
the change of variables $y = 2^{j}u-k$
\begin{eqnarray}\label{eq:1}
&& 2^{-j H(t)} |\varepsilon_{j,k}| \nonumber \\
& \leq &  2^{j}\int_{\R} 
\big|B(u, H(t)) - B(t, H(t))\big| \big|\Psi(2^{j}u-k, -H(t)-1)\big|  du \nonumber\\
& =  & \int_{|y|\leq 2^{j/2} }
       \big|B(\frac{k+y}{2^{j}}, H(t)) - B(\frac{k+y}{2^{j}} , H(\frac{k+y}{2^{j}}))\big| \big|\Psi(y, -H(t)-1)\big|  dy \nonumber \\ 
  &&+  \int_{|y|\leq 2^{j/2} }
       \big|B(\frac{k+y}{2^{j}}, H(\frac{k+y}{2^{j}})) - B(t,
     H(t))\big| \big|\Psi(y, -H(t)-1)\big|  dy\nonumber \\
&& + \int_{ |y|> 2^{j/2} }\big|B(\frac{k+y}{2^{j}}, H(t)) - B(t, H(t))\big|
   \big|\Psi(y, -H(t)-1)\big| dy  
\end{eqnarray}
where $k := k_{j}(t)$.  
Let us now provide an appropriate
upper bound for each of term in the right-hand side of
\eqref{eq:1}. Note that the assumption of regularity  on $H$ implies that there is  a neighborhood $I$
of $t$ and a constant $c_{0}>0$ such that 
\begin{equation}\label{eq:regH}
|H(s)-H(t)| \leq c_{0} |s-t|^{\gamma} \quad \forall t \in I.  
\end{equation}
Now, for the first term we notice that  
\begin{equation}\label{eq:distt}
  |t-\frac{k+y}{2^{j}}| \leq 2^{-j} |y + (k-2^{j}t)|\leq 2^{-j}(|y|+1) \end{equation}
and in particular if $|y|\leq 2^{j/2}$, then $\frac{k+y}{2^{j}} \in I$
for large $j$. It follows that 
\begin{eqnarray}\label{eq:2}
&& \int_{|y|\leq 2^{j/2} }
       \big|B(\frac{k+y}{2^{j}}, H(t)) - B(\frac{k+y}{2^{j}}, H(\frac{k+y}{2^{j}}))\big| \big|\Psi(y, -H(t)-1)\big|  dy \nonumber \\ 
  & \leq & c_{I,1}C_{1} \Big(2^{j}\int_{|y|\leq 2^{j/2} }
\big|H(t) - H(\frac{k+y}{2^{j}}) \big| \big|\Psi(y, -H(t)-1)\big|
           dy \nonumber \\
  & \leq & C_{2} \Big(2^{j}\int_{|y|\leq 2^{j/2} }
\big|t - \frac{k+y}{2^{j}} \big|^{\gamma} \big|\Psi(y, -H(t)-1)\big|
           dy \nonumber \\
 & \leq & C_{2} 2^{-\gamma j} \int_{\R} \big(1+ |y| \big)^{\gamma} \big|\Psi(y, -H(t)-1)\big|
          dy \nonumber \\
  &\leq  & C_{3} 2^{-\gamma j} 
\end{eqnarray}
for some positive random constants $C_{2}, C_{3}$, by using
successively  
Lemma \ref{lem:fromAB}, equations \eqref{eq:regH}, \eqref{eq:distt} and \eqref{eqn:fastdecay}.

For the second term, if $|y|\leq 2^{j/2}$, inequality \eqref{eq:distt}
gives the existence of a constant $c>0$ such that $
|t-\frac{k+y}{2^{j}}|\leq c 2^{-j/2}$. Hence
\begin{eqnarray}\label{eq:3}
 && \int_{|y|\leq 2^{j/2} }
       \big|B(\frac{k+y}{2^{j}}, H(\frac{k+y}{2^{j}})) - B(t,
     H(t))\big| \big|\Psi(y, -H(t)-1)\big|  dy  \nonumber \\
& \leq &   \sup\left\{ 
\frac{|B_{H}(s) - B_{H} (t)|}{ \sigma(|s-t|)} :  |s-t| < c
         2^{-j/2}\right\} \int_{|y|\leq 2^{j/2} }
         \sigma(|t-\frac{k+y}{2^{j}}|) \big|\Psi(y, -H(t)-1)\big| dy \nonumber \\
  & \leq & \sup\left\{ 
\frac{|B_{H}(s) - B_{H} (t)|}{ \sigma(|s-t|)} :  |s-t| < c
         2^{-j/2}\right\} \sigma(2^{-j})\int_{\R }
           \sigma(|y|+1) \big|\Psi(y, -H(t)-1)\big| dy \nonumber \\
 & \leq &c_{2 }\sup\left\{ 
\frac{|B_{H}(s) - B_{H} (t)|}{ \sigma(|s-t|)} :  |s-t| < c
         2^{-j/2}\right\} \sigma(2^{-j})
\end{eqnarray}
for a constant $c_{2}>0$, using  \eqref{eq:distt}, the
submultiplicativity property of $\sigma$, 
\eqref{eqn:fastdecay} and the polynomial growth of $\sigma$.

The upper bound for the last term is obtained using again the fast
decay \eqref{eqn:fastdecay} of
the wavelet for $L \geq 2 \gamma$ together with the boundedness of the process $B$. Indeed,
on can write 
\begin{eqnarray}\label{eq:4}
&&\int_{ |y|> 2^{j/2} }\big|B(\frac{k+y}{2^{j}}, H(t)) - B(t, H(t))\big|
   \big|\Psi(y, -H(t)-1)\big| dy  \nonumber  \\
& \leq & C_{3} \int_{ |y|> 2^{j/2} }
         \frac{1}{(1+|y|)^{2L}} du \nonumber\\
& \leq & C_{3} 2^{-Lj/2}\int_{\R }\frac{1}{(1+|y|)^{L}} dy\nonumber\\
& \leq & C_{3}' 2^{-\gamma j}
\end{eqnarray}
for some a positive random constant $C_{3},C_{3}'$. Putting together equations
\eqref{eq:1}, \eqref{eq:2}, \eqref{eq:3} and
\eqref{eq:4}  leads to the conclusion.
\end{proof}

In order to prove Theorem \ref{thm:optimal}, it suffices now to
provide convenient asymptotic lower bounds for the coefficients
$\varepsilon_{j,k}$. We summarize the useful known results of
\cite{MR3839281}, \cite{ayacheesserkleyntssens} and
\cite{esserloosveldt} in the following Lemma.

\begin{Lemma}\label{lem:lowerbound}\cite{MR3839281,ayacheesserkleyntssens,esserloosveldt}
 Let $(\varepsilon_{j,k})_{(j,k)\in \Z^{2}}$ be a sequence of independent 
 $\mathcal{N}(0,1)$ random variables.  There exists an event
  $\Omega^*_{2}$ of probability $1$  on which
  \begin{enumerate}
  \item for every $t \in \R$, one has
$$
\limsup_{j \to + \infty} |\varepsilon_{j, k_{j}(t)} |
\geq 2^{-3/2} \sqrt{\pi} \, , 
$$
  \item for every non-empty open interval $I$ of $\R$, there is $t\in I$ such that 
\begin{equation*}
\limsup_{j \to + \infty } \frac{ |\varepsilon_{j,k_{j}(t)}| }{ \sqrt{j}}  >0 \, ,
\end{equation*}
  \item for almost every $t \in \R$, one has
\begin{equation*}\label{eq:lowerbound2}
\limsup_{j \to + \infty }  \dfrac{|\varepsilon_{j, k_{j}(t)}|
}{\sqrt{\log j}} >0 \, .
\end{equation*}
  \end{enumerate}
\end{Lemma}

 The proof of the main result of this section is now straightforward. 
 
 \begin{proof}[Proof of Theorem \ref{thm:optimal}]
From Theorem
\ref{thm:slow}, equation  \eqref{eqn:limord}  and equation
\eqref{eqn:limrap}, it suffices to prove the three lower bounds. 

Let us work on the event $\Omega^\ast \cap \Omega^\ast_2$ of
probability $1$. In each case, if $\sigma$ denotes the corresponding
 modulus of continuity, we know from Proposition
  \ref{prop:controlecoeff} and  Remark \ref{rem:submult} that 
$$
2^{-j H(t)} |\varepsilon_{j,k_{j}(t)}| \leq C \Big(\sup\left\{ 
\frac{|B_{H}(s) - B_{H} (t)|}{ \sigma(|s-t|)} :  |s-t| < c   2^{-j/2}\right\}
 \sigma(2^{-j}) +
2^{-\gamma j }\Big)
$$
with $\gamma > H(t)$ by Condition \ref{condii}.  Lemma
\ref{lem:lowerbound} then implies that
$$
0 < \limsup_{j \to +\infty} \frac{
  |\varepsilon_{j,k_{j}(t)}|}{\sigma(2^{-j})} \leq C \lim_{j \to +
  \infty }\sup\left\{ 
\frac{|B_{H}(s) - B_{H} (t)|}{ \sigma(|s-t|)} :  |s-t| < c   2^{-j/2}\right\} 
$$
since $\frac{2^{-\gamma j}}{\sigma(2^{-j})} $ tends to $0$ as $j$ tends to
infinity. 
\end{proof}


\section{Extensions}

The methodology developed in the previous sections can easily be
adapted to study very general random wavelet series of the form
$$
 f_H = \sum_{j \in \N} \sum_{k \in \Z} \varepsilon_{j,k} 2^{-H(k2^{-j})j} \psi(2^j \cdot -k )
$$
where $(\varepsilon_{j,k})_{(j,k)\in \Z^{2}}$ still denotes a sequence of
i.i.d. $\mathcal{N}(0,1)$ random variables. Amoung the families
of wavelet basis that exist, we will work with two classes:
The Lemari\'e-Meyer wavelets for which  
 $\psi$ belong to the Schwartz class $\mathcal{S}(\R)$, or Daubechies wavelets for which   $\psi$ is a 
compactly supported function (see \cite{Daubechies:92}). In both cases, the first moment of the wavelet $\psi$ vanishes.
We will also include the setting given by biorthogonal wavelet basis, \cite{MR1161250,MR1162365}.

Clearly,  as soon as we work with a compactly supported wavelet or a wavelet which decays is sufficiently fast, one can make sure that the function $f_H$ is almost surely well-defined, exploiting Lemma \ref{boundnormal}. 
The process $f_H$ gives a multifractal version of the random series studied in
\cite{esserloosveldt}, by substituting the exponent $h$
at level $(j,k)$  by $H(k2^{-j})$ as done in
\cite{bbcI00}. Of course, this model can not be used to represent MBM. Nevertheless, we believe that it can have its own interest as it can be used to numerically simulate multifractional signals more efficiently than by considering the random series \eqref{eqn:waveletdecomp} since one can avoid the computation of the fractional primitives. Moreover, concerning the pointwise regularity, we will show that we do not alter the results obtained in the previous sections.

The biorthogonality of the wavelets allows to state in our present context the
following resut similar to  Proposition \ref{prop:controleWcoeff}.

\begin{Prop}\label{prop:controleWcoeff}
Let us consider $t \in \R$ and a submultiplicative modulus of
continuity $\sigma$. 
Assume that
$$
f= \sum_{j \in \Z}\sum_{k \in \Z}c_{j,k} \psi(2^{j}\cdot -k)
$$
is a bounded function and that the wavelet $\psi$  satisfies
$$
\sup_{y \in \R} (1+|y|)^{2L} |\psi(y)| < + \infty
$$
for some $L>0$, and
$$
\int_{\R} \sigma(1+|y|) |\psi(y)| dy < +\infty. 
$$
Then for every $j$ large enough, one has
$$
 |c_{j,k_{j}(t)}| \leq c \Big(\sup\left\{ 
\frac{|f(s) - f(t)|}{ \sigma(|s-t|)} :  |s-t| < c   2^{-j/2}\right\}
 \sigma(2^{-j}) +
2^{-L j/2 }\Big)
 $$
for a constant $c>0$, where the supremum may take the value $+ \infty$. 
\end{Prop}

\begin{proof}
  The (bi)orthogonality of the  wavelets allows to write
  \[
c_{j,k}=2^{j}\int_{\R}f(x) \psi(2^ju-k)\, du.
\]
Using similar arguments as in Proposition \ref{prop:controlecoeff},
for $k=k_j(t)$, we can write
\begin{eqnarray*}\label{eq:1bis}
  |c_{j,k}| 
& \leq &  2^{j}\int_{\R} 
\big|f(u)- f(t)\big| \big|\psi(2^{j}u-k)\big|  du \\
& =  & \int_{|y|\leq 2^{j/2} }
       \big|f(\frac{k+y}{2^{j}}) - f(t)\big| \big|\psi(y)\big|  dy + \int_{ |y|> 2^{j/2} }\big|f(\frac{k+y}{2^{j}}) - f(t)\big|
   \big|\psi(y)\big| dy  \\
 & \leq & \sup\left\{ 
\frac{|f(s) - f(t)|}{ \sigma(|s-t|)} :  |s-t| < c   2^{-j/2}\right\}
          \sigma(2^{-j})\int_{\R} \sigma(1+|y|)  \big|\psi(y)\big|
          dy\\
  &&          + 2\|f\|_{\infty} 2^{- Lj/2} \int_{\R} \frac{1}{(1+|y|)^{L}}dy
\end{eqnarray*}
hence the conclusion.
  \end{proof}

This section aims at showing that $f_{H}$ still shares the same features as MBM
when one considers its pointwise regularity. Moreover, in this context, we can significantly reduce the condition made on the regularity of the function $H$ to obtain the results. In the sequel, Condition \ref{condi} is replaced by the following.

\begin{Cond}\label{condibis}
The Hurst function $H \, : \, \R \to [a,b]$, with $0<a<b<1$, is such that for all $t \in \R$ there exist $R_t>0$ and $c_t>0$ such that
\[ |H(s)-H(t)| \leq \frac{c_t}{\log |s-t|^{-1}} \]
for all $s \in \R$ with $|s-t| \leq R_t$.
\end{Cond}

\begin{Rmk}
Of course, any function $H$ satisfying Condition \ref{condibis} is necessarily continuous and any H\"older-continuous function satisfies Condition \ref{condibis}. In particular, Condition \ref{condibis} is weaker than Condition \ref{condi}.
\end{Rmk}

\begin{Rmk}
 In \cite{MR1626706}, it is proved that if $H$ is
the function ``H\"older exponent'' of a continuous function, then there
exists a sequence $(P_j)_{j \in \N_{0}}$ of polynomials such that
\begin{equation}\label{eq:hmeyer}
 \left\{
\begin{array}{l}
H(t) = \liminf_{j \to + \infty} P_j(t) \\[2ex] 
\ \| D P_j \|_\infty \leq j  , \quad \forall j \in \N_{0}. \\ 
\end{array} 
\right. 
\end{equation}
Because of Condition \ref{condibis}, our function $H$ is not as general, but if a function $H$ satisfies  \eqref{eq:hmeyer} and if we assume the existence of a constant $C>0$ such that, for all $t \in \R$ and $j \in \N_{0}$,
\begin{equation}\label{eq:convergencePJ}
|H(t)-P_j(t)| \leq c_t j^{-1}
\end{equation}
then Condition \ref{condibis} is satisfied. Note that the authors in \cite{MR4398453} assume a condition similar to \eqref{eq:convergencePJ} to prove a law of the iterated logarithm for a multifractional extension of BM defined using the Faber-Schauder base.
\end{Rmk}

Our result concerning the pointwise regularity of the process $f_{H}$ can
then be stated as follow.

\begin{Thm}\label{thm:main2}
If the function $H \, : \, \R \to [a,b]$ satisfies the Condition \ref{condibis} then almost surely, for every non-empty interval $I$ of $\R$,
\begin{itemize}
\item  there exists $t \in I$  such that
\begin{equation}\label{eqn:rapid1bis}
0<\limsup_{s \to t} \frac{|f_{H}(s)-f_{H}(t)|}{|s-t|^{H(t)} \sqrt{\log|s-t|^{-1}}} < \infty,
\end{equation}
\item almost every point $t \in I$ is such that
\begin{equation}\label{eqn:ord1bis}
0<\limsup_{s \to t} \frac{|f_{H}(s)-f_{H}(t)|}{|s-t|^{H(t)} \sqrt{\log\log|s-t|^{-1}}} < \infty,
\end{equation}
\item  there exists $t \in I$ such that
\begin{equation}\label{eqn:slow1}
0<\limsup_{s \to t} \frac{|f_{H}(s)-f_{H}(t)|}{|s-t|^{H(t)}} < \infty.
\end{equation}
\end{itemize}
\end{Thm}

\begin{proof}
  We will slightly modify the proofs  done in the
previous sections for the MBM. As previously, it suffices to work on $[0,1)$. Let us first focus on the three upper
bounds. On this purpose, for all $j \in \N$, we define the random series
\[ {f}_{H,j} := \sum_{k \in \Z} 2^{-j H(k2^{-j})} \varepsilon_{j,k} \psi(2^j \cdot - k). \]
Let us start by showing the existence of slow points
\eqref{eqn:slow1}. As previously, we take $m \in \N$ such that
$\frac{1}{m}<a$ and on an event of probability $1$, Theorem \ref{thm:procedure} allows to consider  $t
\in S_{\text{low,m}}^\mu$, for some $\mu >0$. Now, if
$s \in (0,1)$ is such that $2^{-n} \leq |s-t| \leq 2^{-n+1}$, we write
\begin{align*}
\left|f_{H}(t)-f_{H}(s)\right| & \leq   \left|\sum_{j=0}^n ( f_{H,j}(t)-f_{H,j}(s)) \right|  + \sum_{j \geq n+1}  \left|f_{H,j}j(t)\right|  + \sum_{j \geq n+1}  \left| f_{H,j}(s) \right|. 
\end{align*}
As in Lemma \ref{lemma1}, we have
\begin{align*}
\left|\sum_{j=0}^n ( f_{H,j}(t)-f_{H,j}(s)) \right| & \leq |s-t| \left(\sum_{j=0}^n \sum_{k \in \Z} |\varepsilon_{j,k}|  2^{j(1- H(k2^{-j}))} |D_t \psi(2^jx-k)| \right)
\end{align*}
for some $x$ between $s$ and $t$. Then, similarly to \eqref{decomposomme}, we deduce, using the fast decay property \eqref{eqn:fastdecay},
\[ \sum_{j=0}^n \sum_{k \in \Z} |\varepsilon_{j,k}|  2^{j(1- H(k2^{-j}))} |D_t \psi(2^jx-k)|  \leq c_1 \mu \sum_{j=0}^n \sum_{k \in \Z} \frac{2^{j(1- H(k2^{-j}))}}{(3+|2^jx-k|)^L},\]
for $L$ sufficiently large and whose value will be specified later and a deterministic positive constants $c_1$ which only depends on $\psi$, $L$ and $m$. Now, if $k \in \Z$ is such that $|t-k2^{-j}| \leq 2^{-j/2}$ then, by Condition \ref{condibis}, $|H(t)-H(k2^{-j})| \leq 2 c_t j^{-1}$ and
\[ \frac{2^{j(1- H(k2^{-j}))}}{(3+|2^jx-k|)^L} \leq 2^{j(1-H(t))} \frac{2^{2 c_t}}{(3+|2^jx-k|)^L}.  \]
On the other hand, if $|t-k2^{-j}| > 2^{-j/2}$, we write
\[ \frac{2^{j(1- H(k2^{-j}))}}{(3+|2^jx-k|)^L} \leq 2^{j(1-H(t))} \frac{2^{j(b-a)}}{(3+|2^jx-k|)^L}  \]
and then, if $L \geq 2$ is such $b-a< L/2-1$, as $|2^jx-k|> 2^{j/2}-2$, we get
\[ \frac{2^{j(1- H(k2^{-j}))}}{(3+|2^jx-k|)^L} \leq 2^{j(1-H(t))} \frac{1}{(3+|2^jx-k|)^2} . \]
In total, we obtain
\begin{align*}
\left|\sum_{j=0}^n ( f_{H,j}(t)-f_{H,j}(s)) \right| & \leq c_2 \mu |t-s|  \sum_{j=0}^n 2^{j(1-H(t))} \sum_{k \in \Z} \frac{1}{(3+|2^jx-k|)^2} \\
& \leq c_3 \mu |t-s| \sum_{j=0}^n 2^{j(1-H(t))} \\
& \leq c_4 \mu |t-s|  2^{n(1-H(t))} \\
& \leq c_5 |t-s|^{H(t)},
\end{align*}
where $c_2$, $c_3$, $c_4$ and $c_5$ are deterministic positive constants not depending on any relevant quantity. Modifying the proofs of Lemma \ref{lemma2} and Theorem \ref{thm:slow} in exactly the same way, we obtain
\[\sum_{j \geq n+1}  \left|f_{H,j}(t)\right| \leq  c_6 |t-s|^{H(t)}\]
and, as $|s-t| \leq 2^{-n+1}$, by Condition \ref{condibis},
\begin{align*}
\sum_{j \geq n+1}  \left|f_{H,j}(s)\right| & \leq c_6 |t-s|^{H(t)} 2^{|H(t)-H(s)|n} \leq 2^{2c_t}c_6 |t-s|^{H(t)}, 
\end{align*}
with $c_6$ a deterministic positive constants which does not depend on any relevant quantity.

Inequalities \eqref{eqn:rapid1bis} and \eqref{eqn:ord1bis} are proved in a
similar way, exploiting the alternative arguments given in Remark
\ref{rmk:rapord}.

The lower bounds are obtained by combining Lemma \ref{lem:lowerbound}
together with 
Proposition \ref{prop:controleWcoeff}.
\end{proof}

\begin{Rmk}
In the particular case where the function $H$ is constant, we recover
the random wavelet series studied in \cite{esserloosveldt}. Note
however that, even in this simple case, we improve here \cite[Theorem
2.4]{esserloosveldt} since we obtain that the three above limsup are
strictly positive, even if the wavelet is \textit{not} compactly supported, see \cite[Remark 5.2]{esserloosveldt}. 
\end{Rmk}

\begin{Rmk}
A careful look at the proofs shows that the random series $f_H$ could also be defined through a biorthogonal system of vaguelets, see \cite{MR1325535,Meyer:97,MR3119252}. Recall that a family of functions $\Psi_{j,k}$ is called vaguelets if it satisfies a property of localization
$$
|\Psi_{j,k}(t)|\leq C 2^{j/2} (1+|2^jt-k|)^{-1-\alpha_1} \quad \forall \, t \in \R,
$$ 
a property of oscillation
$$
\int_{\R} \Psi_{j,k}(t) dt =0
$$
and a property of regularity
$$
|\Psi_{j,k}(t)- \Psi_{j,k}(s)| \leq C 2^{j(\alpha_2+1/2)}|t-s|^{\alpha_2} \quad \forall \, s,t \in \R
$$
for some constant $C>0$ and $0<\alpha_2<\alpha_1<1$.
\end{Rmk}


Let us end this section by considering a third process. 
In \cite{MR2743002}, the authors have proved that the process $\{Z(t) \, : \, t \in \R \}$ defined for each $t \in \R$ as
\begin{equation} \label{eqn:toconsiderforZ}
Z(t) = \sum_{j \in \Z} \sum_{k \in \Z} 2^{-jH(k2^{-j})} \varepsilon_{j,k} \left( \Psi(2^jt-k,H(k 2^{-j}))-\Psi(-k,H(k 2^{-j}) \right)\end{equation}
has common property with MBM. Namely,
\begin{enumerate}[(a)]
\item when the function $H$ is constant, it reduces to FBM.
\item the trajectories of the process
\[ \overline{Z}(t):= \sum_{j=- \infty}^{-1} \sum_{k \in \Z} 2^{-jH(k2^{-j})} \varepsilon_{j,k} \left( \Psi(2^jt-k,H(k 2^{-j}))-\Psi(-k,H(k 2^{-j}) \right) \]
are almost surely $C^\infty$ functions. Thus, the regularity of $Z$ is only determined by the process
\begin{equation}\label{eqn:tildeZ}
\widetilde{Z}(t):= \sum_{j\in \N} \sum_{k \in \Z} 2^{-jH(k2^{-j})} \varepsilon_{j,k} \left( \Psi(2^jt-k,H(k 2^{-j}))-\Psi(-k,H(k 2^{-j}) \right). 
\end{equation}
\item the process $Z$ is also locally asymptotically self-similar.
\item almost surely, for all $t$, the pointwise H\"older exponent of $Z$ at $t$ is $H(t)$.
\item if $a$ and $b$ satisfy the condition
\begin{equation}\label{eqn:conditionpourZ}
1-b > (1-a)(1-ab^{-1})
\end{equation}
then there exists an exponent $d \in (b,1]$ such that, almost surely, the process $Z-B_H$ is uniformly H\"older of exponent $d$. In other words, there exists a process $X$ more regular than $B_H$ and $Z$ such that
\[B_H=Z + X. \]
\end{enumerate}
In some sense, all these facts mean that, up to an additive regular
process, if condition \eqref{eqn:conditionpourZ} holds, $Z$ is an
appropriate representation of MBM. Of course, even if condition
\eqref{eqn:conditionpourZ} does not hold, the process $Z$ has its own
interest. Here, we want to show that $Z$ still shares the same
features as MBM when one considers its pointwise regularity under the
less restrictive condition \ref{condibis}. 

\begin{Thm}\label{thm:tripoint1}
If the function $H \, : \, \R \to [a,b]$ satisfies the Condition \ref{condibis} then almost surely, for every non-empty interval $I$ of $\R$,
\begin{itemize}
\item  there exists $t \in I$ such that
\begin{equation}\label{eqn:rapid2bis}
\limsup_{s \to t} \frac{|Z(s)-Z(t)|}{|s-t|^{H(t)} \sqrt{\log|s-t|^{-1}}} < \infty.
\end{equation}
\item almost every point $t \in I$ is such that
\begin{equation}\label{eqn:ord2}
\limsup_{s \to t} \frac{|Z(s)-Z(t)|}{|s-t|^{H(t)} \sqrt{\log\log|s-t|^{-1}}} < \infty.
\end{equation}
\item  there exists $t \in I$ such that
\begin{equation}\label{eqn:slow2bis}
\limsup_{s \to t} \frac{|Z(s)-Z(t)|}{|s-t|^{H(t)}} < \infty.
\end{equation}

\end{itemize}
\end{Thm}

\begin{proof}
We already know that it suffices to prove \eqref{eqn:rapid2bis}, \eqref{eqn:ord2} and \eqref{eqn:slow2bis} for the process $\widetilde{Z}$. Also, when one considers the increments $\widetilde{Z}(t)-\widetilde{Z}(s)$, $t,s \in \R$, the terms $ 2^{-jH(k2^{-j})}\Psi(-k,H(k2^{-j}))$ in \eqref{eqn:tildeZ} cancel and thus we just need to study the pointwise regularity of the random series
\[\sum_{j \in \N} \sum_{k \in \Z} 2^{-j H(k2^{-j})} \varepsilon_{j,k} \Psi(2^j \cdot-k,H(k2^{-j})).
\]
Then, it suffices  to replace $\psi$ by $\Psi(\cdot, H(k2^{-j}))$ in
the proof of Theorem \ref{thm:main2}. 
\end{proof}

When one wants to prove the positiveness of the limits in Theorems \ref{thm:optimal} and \ref{thm:main2}, the
strategy is to consider the biorthogonality property of the basis to
express the coefficients in terms of the increments of the
process. In the present situation, it seems that there is no obvious connexion between the random coefficients in the series \eqref{eqn:toconsiderforZ}
 and the oscillations of the process. In particular, one can not apply this strategy anymore. Therefore, the positiveness of the three
 limits remains an interesting open question which needs different
 tools than those developed in this paper.
 %


\bigskip

\bibliography{biblio}{}
\bibliographystyle{plain}

\end{document}